\documentclass{llncs}
\usepackage{amsmath}
\usepackage{amsfonts}
\usepackage{euscript}

\newcommand{\beeq}{\begin{eqnarray*}}
\newcommand{\eneq}{\end{eqnarray*}}

\newcommand{\R}{\mathbb{R}}
\newcommand{\cF}{\mathcal{F}}

\begin{document}
\title{ Witness sets}
\author{G\'erard Cohen\inst{1}\and Hugues Randriam\inst{1} \and Gilles Z\'emor\inst{2}}
\institute{Ecole Nationale Sup\'erieure des
T\'el\'ecommunications,\\ 
46 rue Barrault,\\
75 634 Paris 13, France\\
\email{cohen@enst.fr, randriam@enst.fr}
\and 
Institut de Math\'ematiques de Bordeaux, \\
Universit\'e de Bordeaux, UMR 5251,\\ 351 cours de la Lib\'eration,\\
33405 Talence, France. \\
\email{Gilles.Zemor@math.u-bordeaux1.fr}}

\maketitle


\begin{abstract}
Given a set $C$ of binary $n$-tuples and $c \in C$, how many bits of $c$
suffice to distinguish it from the other elements in $C$~?
We shed new light on this old combinatorial problem and improve on previously
known bounds.
\end{abstract}

\section{Introduction}
Let $C\subset\{0,1\}^n$ be a set of distinct binary vectors that we
will call a code, and denote by $[n] = \{1,2,...n\}$ the set of
coordinate positions. It is standard in coding theory to ask for codes
(or sets) $C$ such that every codeword $c\in C$ is as different as
possible from all the other codewords. The most usual interpretation of
this is that every codeword $c$ has a large Hamming distance to all
other codewords, and the associated combinatorial question is to
determine the maximum size of a code that has a given minimal Hamming
distance $d$.
The point of view of the present paper is
to consider that ``a codeword $c$ is as different
as possible from all the other codewords'' means that there exists a
small subset $W\subset [n]$ of coordinates such that $c$ differs from
every other codeword in $W$. Put differently, it is possible to single out $c$
from all the other codewords by focusing attention on a small subset
of coordinates. 
More precisely, for $x\in \{0,1\}^n$, and $W\subset [n]$ let
us define the projection $\pi_W$
\begin{eqnarray*}
  \pi_W~:\{0,1\}^{[n]} & \rightarrow & \{0,1\}^W\\
             x & \mapsto     & (x_i)_{i\in W}
\end{eqnarray*}
and let us say that $W$ is a {\em witness set} (or a witness for
short) for $c\in C$ if
$\pi_W(c)\neq \pi_W(c')$ for every $c'\in C$, $c\neq c'$. Codes
for which every codeword has a small witness set arise in a variety
of contexts, in particular in machine learning theory \cite{ABCS,B,GK}
where a witness set is also called a specifying set or a discriminant:
see \cite[Ch. 12]{J} for a short survey of known results
and also \cite{AH} and references therein for a more recent discussion
of this topic and some variations. 

Let us now say that a code
has the $w$-witness property, or is a {\em $w$-witness code}, if
every one of its codewords has a witness set of size $w$. Our
concern is to study the maximum possible cardinality $f(n,w)$
 of a $w$-witness code of length $n$. We shall give improved upper and
lower bounds on $f(n,w)$ that almost meet.

The paper is organised as follows. Section~\ref{sec:easy}
gives some easy facts for reference. Section \ref{sec:improved}
is devoted to upper bounds on $f(n,w)$ and introduces our main result,
namely Theorem~\ref{th:upper}. Section~\ref{sec:constant} is devoted
to constant weight $w$-witness codes, and we derive precise values
of the cardinality of optimal codes. 
Section~\ref{sec:mean}
studies mean values for the number  of witness sets of a codeword and
the number of codewords that have a given witness set.
Section~\ref{sec:constructions} is devoted to constructions of large
$w$-witness codes, sometimes giving improved lower values of $f(n,w)$.
Finally, Section~\ref{sec:problems} concludes with some open problems.

\section{Easy and known facts}\label{sec:easy}
Let us start by mentioning two self-evident facts
\begin{itemize}
\item If $C$ is a $w$-witness code, so is any translate $C+x$,
\item $f(n,w)$ is an increasing function of $n$ and $w$.
\end{itemize}
Continue with the following example. Let $C$ be the set of all
$n$ vectors of length $n$ and weight $1$. Then every codeword of $C$
has a witness of size $1$, namely its support. Note the dramatic
change for the slightly different code $C\cup \{{\bf0}\}$. Now the
all-zero vector ${\bf0}$ has no witness set of size less than $n$.
Bondy \cite{B} shows however that if $|C|\leq n$, then $C$ is
a $w$-witness code with $w\leq |C|-1$ and furthermore $C$ is
a {\em uniform} $w$-witness code, meaning that 
there exists a single
subset of $[n]$ of size $w$ that is a witness set for {\em all}
codewords.

We clearly have the upper bound $|C|\leq 2^w$ for uniform $w$-witness
codes. For ordinary $w$-witness codes however, the best known upper
bound is, \cite[Proposition 12.2]{J},
\begin{equation}
  \label{eq:simple}
  f(n,w)\leq 2^w\binom{n}{w}.
\end{equation}
The proof is simple and consists in applying the pigeon-hole
principle. A subset of $[n]$ can be a witness set for at most $2^w$
codewords and there are at most $\binom{n}{w}$ witness sets.

We also have the following lower bound on $f(n,w)$, based on a trivial
construction of a $w$-witness code.

\begin{proposition}\label{prop:lower}
We have:
${f(n,w)} \geq \binom{n}{w}$.
\end{proposition}

\begin{proof}
Let $C=\binom{[n]}{w}$ be the set of all vectors of weight $w$.
Notice that for all $c\in C$,
$W(c)= support(c)$ is a witness set of $c$.
\end{proof}

Note that the problem is essentially solved for $ w \ge n/2$;
since $f(n,w)$ is increasing with $w$, we then have:

$ 2^n \ge f(n,w) \ge f(n,n/2) \ge \binom{n}{n/2} \ge 2^n/(2n)^{1/2}$.

We shall therefore focus in the sequel on the case $w \le n/2$.

In the next section we improve the upper bound \eqref{eq:simple} to a
quantity that comes close to the lower bound of Proposition~\ref{prop:lower}.

\section{An improved upper bound}\label{sec:improved}

The key result is the following.

\begin{theorem}
\label{keyprop}
Let $g(n,w)=f(n,w)/\binom{n}{w}$.
Then, for fixed $w$, $g(n,w)$ is a decreasing function of $n$.
That is:
$$n\geq v\geq w\qquad\Rightarrow\qquad g(n,w)\leq g(v,w).$$  
\end{theorem}

\begin{proof}
Let $C$ be a binary code of length $n$ having the $w$-witness
property, with maximal cardinality $|C|=f(n,w)$.
Fix a choice function $\phi:C\to\binom{[n]}{w}$
such that for any $c\in C$, $\phi(c)$ is a witness for $c$.
For any $V\in\binom{[n]}{v}$, denote by $C_V$ the subset
of $C$ formed by the $c$ satisfying $\phi(c)\subset V$.
Remark that the projection $\pi_V$ is injective on $C_V$, since each element
of $C_V$ has a witness in $V$. Then $\pi_V(C_V)$ also
has the $w$-witness property.

Remark now that if $V$ is uniformly distributed in $\binom{[n]}{v}$
and $W$ is uniformly distributed in $\binom{[n]}{w}$ and independent
from $V$, then for any function $\psi:\binom{[n]}{w}\to\R$ one has
  \begin{equation}
  \label{esp.cond.}
  E_W(\psi(W))=E_V(E_W(\psi(W)\,|\,W\subset V)),
  \end{equation}
  where we denote by $E_W(\psi(W))$ the mean value (or expectation)
  of $\psi(W)$ as $W$ varies in $\binom{[n]}{w}$, and so on.

We apply this with $\psi(W)=|\phi^{-1}(W)|$ to find
\begin{equation*}
\begin{split}
g(n,w)&=\binom{n}{w}^{-1}|C|=\binom{n}{w}^{-1}\sum\nolimits_{W\in\binom{[n]}{w}}|\phi^{-1}(W)|\\
&=E_W(\:|\phi^{-1}(W)|\:)\\
&=E_V(E_W(\:|\phi^{-1}(W)|\:|\:W\subset V))\\
&=E_V\left(\binom{v}{w}^{-1}\sum\nolimits_{W\in\binom{V}{w}}|\phi^{-1}(W)|\right)\\
&=E_V\left(\binom{v}{w}^{-1}|C_V|\right)\\
&=E_V\left(\binom{v}{w}^{-1}|\pi_V(C_V)|\right)\\
&\leq g(v,w)
\end{split}
\end{equation*}
the last inequality
because $\pi_V(C_V)$ is a binary code of length $v$ having the $w$-witness
property.
\end{proof}

\noindent
{\bf Remark:}
It would be interesting to try to improve Theorem~\ref{keyprop} using some
  unexploited aspects of the above proof, such as the fact that 
  the choice function $\phi$ may be non-unique, or the fact that 
  the last inequality not only holds in mean value, but for all $V$.
  For instance, suppose there is a codeword
  $c\in C$ (with $C$ optimal as in the proof)
  that admits two distinct witnesses $W$ and $W'$,
  with $W\not\subset W'$.
  Let $\phi$ be a choice function with $\phi(c)=W$,
  and let $\phi'$ be the choice function that coincides
  everywhere with $\phi$, except for $\phi'(c)=W'$.
  Let $V$ contain $W'$ but not $W$. If we denote
  by $C_V'$ the subcode obtained as $C_V$ but using $\phi'$
  as choice function, then $C_V'=C_V\cup\{c\}$ (disjoint
  union), so $|\pi_V(C_V)|=|\pi_V(C_V')|-1<f(v,w)$,
  and $g(n,w)<g(v,w)$.

\medskip

Theorem \ref{keyprop} has a number of consequences: the following
is straightforward.

\begin{corollary}
For fixed $w$, the limit
$$\lim_{n \rightarrow \infty} g(n,w) =\frac{f(n,w)}{\binom{n}{w}}$$
 exists.
\end{corollary}

The following theorem gives an improved upper bound on $f(n,w)$.

\begin{theorem}\label{th:upper}
 For $w\leq n/2$, we have the upper bound:
$$f(n,w) \leq 2w^{1/2}\binom{n}{w}.$$
\end{theorem}

\begin{proof}
 Choose $v=2w$ and use $f(v,w) \leq 2^v$; then
$f(n,w) \leq {\binom{n}{w}} f(2w,w)/{\binom{2w}{w}}$ and the result 
follows by Stirling's approximation.
\end{proof}

Set $w= \omega n$ and denote by
$h(x)$ the binary entropy function
$$h(x) = -x \log_2 x - (1-x) \log_2 (1-x).$$
Theorem \ref{th:upper} together with
Proposition~\ref{prop:lower} yield:

\begin{corollary}\label{asympt}
We have
$$
\begin{array}{rll}
  \lim_{n \rightarrow \infty}  \frac 1n \log_2 f(n,\omega n) &=
  h(\omega) \hspace{1cm}&\text{for}\;\; 0\leq\omega \leq 1/2\\
  & = 1 &\text{for}\;\; 1/2\leq\omega \leq 1.
\end{array}
$$
\end{corollary}

\section{Constant-weight codes}\label{sec:constant}
Denote now by
$f(n,w,k)$ the maximal size of a $w$-witness code with codewords of weight $k$.
The following result is proved using a folklore method usually attributed
to Bassalygo and Elias, valid when the required property is invariant
under some group operation.

\begin{proposition}
We have:
$$\max_{k} f(n,w,k) \leq f(n,w) \leq 
\min_{k} \frac{f(n,w,k) 2^n}{\binom{n}{k}}.$$
\end{proposition}

\begin{proof}
The lower bound is trivial.

For the upper bound, fix $k$, pick an optimal $w$-witness code $C$ and 
consider its $2^n$ translates by all possible vectors.
Every $n$-tuple, in particular those of weight $k$, occurs exactly $|C|$
times in the union of the translates; hence there exists a translate 
(also an optimal $w$-witness code of size $f(n,w)$ - see the remark
at the beginning of Section \ref{sec:easy}) 
containing
at least the average number $|C|{\binom{n}{k}}2^{-n}$ of vectors of weight $k$.
Since $k$ was arbitrary, the result follows.
\end{proof}

We now deduce from the previous proposition the exact value of the function
$f(n,w,k)$ in some cases.

\begin{corollary}
For constant-weight codes we have:
  \begin{itemize}
  \item If $k\leq w \leq n/2$ then $f(n,w,k)= \binom{n}{k}$ and
an optimal code is given by $S_k({\bf0})$,  
the Hamming sphere of radius $k$ centered
on ${\bf0}$.
  \item If $n-k \leq w \leq n/2$, then $f(n,w,n-k)= \binom{n}{k}$ and
an optimal code is given by the sphere $S_k({\bf1})$.
  \end{itemize}
\end{corollary}

\begin{proof}
If $k\leq w \leq n/2$, we have the following series of inequalities:

$$\binom{n}{k} \leq f(n,k,k) \leq f(n,w,k) \leq \binom{n}{k}.$$

If $n-k \leq w \leq n/2$, perform wordwise complementation.
\end{proof}


\section{Some mean values}\label{sec:mean}

Let $C$ be a binary code of length $n$ (not necessarily having the $w$-witness
property).
Let 
$$\mathcal{W}_{C,w}:C\to 2^{\binom{[n]}{w}},\;\;\;
\mathcal{W}_{C,w(}c)=\{W\in\binom{[n]}{w}\;:\;\textrm{$W$ is a witness for $c$}\},$$
and symmetrically,
$$\mathcal{C}_{C,w}:\binom{[n]}{w}\to 2^C,\;\;\;
\mathcal{C}_{C,w}(W)=\{c\in C\;:\;\textrm{$W$ is a witness for $c$}\}.$$

Remark that if $C'\subset C$ is a subcode, 
then $\mathcal{W}_{C',w}(c)\supset\mathcal{W}_{C,w}(c)$ for any $c\in C'$,
while $\mathcal{C}_{C',w}(W)\supset\left(C'\cap \mathcal{C}_{C,w}(W)\right)$
for any $W\in\binom{[n]}{w}$.

\begin{lemma}
With these notations, the mean values of $|\mathcal{W}_{C,w}|$ and
$|\mathcal{C}_{C,w}|$ are related by
  $$|C|E_c(|\mathcal{W}_{C,w}(c)|)=\binom{n}{w}E_W(|\mathcal{C}_{C,w}(W)|),$$
  or equivalently
  $$\frac{|C|}{\binom{n}{w}}=\frac{E_W(|\mathcal{C}_{C,w}(W)|)}{E_c(|\mathcal{W}_{C,w}(c)|)}.$$
\end{lemma}
\begin{proof}
Double count the set
$\left\{(W,c)\in\binom{[n]}{w}\times C\;:\;\textrm{$W$ is a witness for $c$}\right\}$.
\end{proof}


Now let $\gamma(C,w)=E_W(|\mathcal{C}_{C,w}(W)|)$ and let
$\gamma^+(n,w)$ be the maximum possible value of $\gamma(C,w)$
for $C$ a binary code of length $n$,
and $\gamma^{++}(n,w)$ be the maximum possible value of $\gamma(C,w)$
for $C$ a binary code of length $n$ having the $w$-witness
property.

\begin{lemma}
With these notations, one has $\gamma^+(n,w)=\gamma^{++}(n,w)$.
\end{lemma}
\begin{proof}
By construction $\gamma^+(n,w)\geq\gamma^{++}(n,w)$. On the other
hand, let $C$ be a binary code of length $n$ with
$\gamma(C,w)=\gamma^+(n,w)$, and let then 
$C'$ be the subcode of $C$ formed by the $c$ having at least one
witness of size $w$, 
\emph{i.e.}
  $C'=\bigcup_{W\in\binom{[n]}{w}}\mathcal{C}_{C,w}(W)$.
Then $C'$ has the $w$-witness property, and
$$\gamma^{++}(n,w)\geq\gamma(C',w)\geq\gamma(C,w)=\gamma^+(n,w).$$
\end{proof}

The technique of the proof of Proposition \ref{keyprop}
immediately adapts to give:
\begin{proposition}
With these notations,
$w$ being fixed, $\gamma^+(n,w)$ is a decreasing function of $n$.
That is:
$$n\geq v\geq w\qquad\Rightarrow\qquad \gamma^+(n,w)\leq \gamma^+(v,w).$$  
\end{proposition}
\begin{proof}
Let $C$ be a binary code of length $n$ with $\gamma(C,w)=\gamma^+(n,w)$.
For $V\in\binom{[n]}{v}$, denote by $C_V$ the subset
of $C$ formed by the $c$ having at least one witness of size $w$ included
in $V$, \emph{i.e.}
  $C'_V=\bigcup_{W\in\binom{V}{w}}\mathcal{C}_{C,w}(W)$.
Then $C'_V$ has the $w$-witness property,
$\mathcal{C}_{C,w}(W)\subset\mathcal{C}_{C'_V,w}(W)$ for any $W\subset V$,
and $\pi_V$ is injective on $C'_V$.
  Using this and \eqref{esp.cond.}, one gets:
\begin{equation*}
\begin{split}
\gamma^+(n,w)&=E_W(|\mathcal{C}_{C,w}(W)|)\\
&=E_V(E_W(\:|\mathcal{C}_{C,w}(W)|\:|\:W\subset V))\\
&\leq E_V(E_W(\:|\mathcal{C}_{C'_V,w}(W)|\:|\:W\subset V))\\
&=E_V(E_W(\:|\mathcal{C}_{\pi_V(C'_V),w}(W)|\:|\:W\subset V))\\
&=E_V(\gamma(\pi_V(C'_V),w))\\
&\leq \gamma^+(v,w).
\end{split}
\end{equation*}
\end{proof}

\section{Constructions}\label{sec:constructions}

\subsection{A generic construction}

Let $\cF\subset\binom{[n]}{\leq w}$
be a set of subsets of $\{1,\dots,n\}$ all
having cardinality at most $w$.

Let $C_\cF\subset\{0,1\}^n$ be the set of words having support
included in one and only one $W\in\cF$.
Then:

\begin{proposition} 
With these notations, $C_\cF$ has the $w$-witness property.
\end{proposition}
\begin{proof} 
For each $c\in C_\cF$, let $W_c$ be the unique  
$W\in\cF$ containing the support of $c$. Then $W_c$ is
a witness for $c$.
\end{proof}

\medskip

\textbf{Example~1.}
For $\cF=\binom{[n]}{w}$
we find $C_\cF=S_w({\bf0})$,
and $$f(n,w)\geq|C_\cF|=\binom{n}{w}.$$

  \textbf{Example~1'.}
  Suppose $w\geq n/2$. Then for $\cF=\binom{[n]}{n/2}$
  we find $C_\cF=S_{n/2}({\bf0})$,
  and $$f(n,w)\geq|C_\cF|=\binom{n}{n/2}$$
  (where for ease of notation we write $n/2$ instead of $\lfloor n/2 \rfloor$).

\textbf{Example~2.}
For $\cF=\{W\}$ with $|W|\leq w$ we find
$C_\cF=\{0,1\}^W$ (where we see $\{0,1\}^W$ as a subset of
$\{0,1\}^{n}$ by extension by $0$ on the other coordinates),
and
$$f(n,w)\geq|C_\cF|=2^w.$$

\textbf{Exemple~3.}
Let $\cF$ be the set of (supports of) words
of a code with constant weight $w$
and minimal distance $d$ (one can suppose $d$ even).
Then for all distinct $W,W'\in\cF$ one has $|W\cap W'|\leq w-d/2$,
so for all $W\in\cF$, the code $C_\cF$ contains 
all words of weight larger than $w-d/2$ supported in $W$. This implies~:

\begin{corollary} 
For all $d$ one has
$$f(n,w)\geq A(n,d,w)B(w,d/2-1)$$
where:
\begin{itemize}
\item $A(n,d,w)$ is the maximal cardinality of a code of length $n$
  with minimal distance at least $d$ and constant weight $w$
\item $B(w,r)=\Sigma_{1 \le i \le r} \binom{w}{i}$ is the cardinality
  of the ball of radius $r$ in $\{0,1\}^w$.
\end{itemize}
\end{corollary} 

For $d=2$, this construction gives the sphere again.
For $d=4$, this gives $f(n,w)\geq (1+w)A(n,d,w)$.
We consider the following special values:
\begin{itemize}
\item $n=4$, $d=4$, $w=2$: $A(4,4,2)=2$
\item $n=8$, $d=4$, $w=4$: $A(8,4,4)=14$
\item $n=12$, $d=4$, $w=6$: $A(12,4,6)=132$
\end{itemize}
the last two being obtained with $\cF$ the Steiner system $S(3,4,8)$
and $S(5,6,12)$ respectively.

The corresponding codes $C_\cF$ have
same cardinality as the sphere
($2\times 3=6$, $14\times 5=70$ and $132\times 7=924$ respectively), 
but they are not translates of a sphere.
Indeed, when $C$ is a (translate of a) sphere with $w=n/2$,
one has $\mathcal{C}_{C,w}(W)=2$ for any window $W\in\binom{[n]}{w}$.
On the other hand, for $C=C_\cF$ as
  above,
one has by construction $\mathcal{C}_{C,w}(W)=w+1$
for $W\in\cF$.

\subsection{Another construction}

Let $D\subset\{0,1\}^w$ be a binary (non-linear) code of
length 
  $w>n/2$
and minimal weight at least $2w-n$.

Let $C_1$ be the code of length $n$ obtained
by taking all words of length $w$ that do not belong to $D$,
and completing them with $0$ on the last $n-w$ coordinates.
Thus $|C_1|=2^w-|D|.$

Let $C_2$ be the code of length $n$ formed by the words $c$ of
weight exactly $w$, and such that the projection of $c$ on
the first $w$ 
  coordinates
belongs to $D$. Thus if $n_k$ is the
number of
codewords of weight $k$ in $D$, one finds
$|C_2|=\sum_k n_k\binom{n-w}{w-k}$.

Now let $C$ be the (disjoint!) union of $C_1$ and $C_2$.
Then $C$ has the $w$-witness property. Indeed, let $c\in C$.
Then if $c\in C_1$, $c$ admits $[w]$ as witness, while
if $c\in C_2$, $c$ admits its support as witness.

As an illustration, let $D$ be the sphere of radius $w-t$ in $\{0,1\}^w$,
  for $t\in\{1,\dots,\frac{n-w}{2}\}$. Then

$$f(n,w)\geq |C|=2^w+\binom{w}{w-t}\left(\binom{n-w}{t}-1\right).$$

  If $w$ satisfies $2^w>\binom{n}{n/2}$ but $w<n-1$, this improves
  on examples 1, 1', and 2 of the last subsection, in that one finds then
  $$f(n,w)\geq |C|>\max(\binom{n}{w},\binom{n}{n/2},2^w).$$

On the other hand, remark that $C_1\subset\{0,1\}^{[w]}$
and $C_2\subset S_w({\bf0})$, so that
$|C|\leq 2^w + \binom{n}{w}$.

\section{Conclusion and open problems}
We have determined the asymptotic size of optimal $w$-witness codes.
A few issues remain open in the non-asymptotic case, among which:
\label{sec:problems}
\begin{itemize}
\item When is the sphere $S_w({\bf0})$ the/an optimal $w$-witness code? 
Do we have $f(n,w) =  \binom{n}{w}$ for $w\leq n/2$~?
In particular do we have $f(2w,w) =  \binom{2w}{w}$~?
\item For $w>n/2$, do we have $f(n,w) \leq  \max(\binom{n}{n/2},2^w + \binom{n}{w})$~?
\item Denoting by
$f(n,w,\ge d)$ the maximal size of a $w$-witness code with minimum
distance $d$, can the asymptotics of Proposition~\ref{asympt} be improved to
$$\frac 1n \log_2 f(n, \omega n, \ge \delta n) < h(\omega)\; ?$$
\end{itemize}

\end{document}